\documentclass[letterpaper, 10 pt, conference]{ieeeconf}  

\newtheorem{theorem}{Theorem}
\newtheorem{lemma}{Lemma}
\newtheorem{remark}{Remark}

\IEEEoverridecommandlockouts                              
\usepackage{epsfig} 
\usepackage{mathptmx} 
\usepackage{times} 
\usepackage{amsmath} 
\usepackage{amssymb}  

\newcommand{\comment}[1]{}

\title{\LARGE \bf
Interval observer for uncertain time-varying SIR-SI epidemiological model of vector-borne disease*}

\author{Maria Soledad Aronna$^{1}$ and Pierre-Alexandre Bliman$^{2}$
\thanks{*This work was supported by Inria, France and CAPES, Brazil (processo 99999.007551/2015-00), in the framework of the STIC AmSud project MOSTICAW. This investigation was also supported by Funda\c c\~ao Get\'ulio Vargas in the framework of the Project of Applied Research entitled ``Controle da Dengue atrav\'es da introdu\c{c}\~ao da bact\'eria {\em Wolbachia}''}
\thanks{$^{1}$Maria Soledad Aronna is with Escola de Matem\'atica Aplicada, Funda\c c\~ao Get\'ulio Vargas, Rio de Janeiro - RJ, Brazil
        {\tt\small soledad.aronna@fgv.br }}%
        \thanks{$^{2}$Pierre-Alexandre Bliman is with Sorbonne Universit\'e, Universit\'e Paris-Diderot SPC, CNRS, Inria, Laboratoire Jacques-Louis Lions, \'equipe Mamba, Paris, France
         and Escola de Matem\'atica Aplicada, Funda\c c\~ao  Get\'ulio  Vargas, Rio de Janeiro - RJ, Brazil
        {\tt\small pierre-alexandre.bliman@inria.fr}}%
}

\begin{document}

\maketitle
\thispagestyle{empty}
\pagestyle{empty}

\def\XOBSa{X_{\text{\eqref{OBS1}}}}
\def\dXOBSa{\dot{X}_{\text{\eqref{OBS1}}}}
\def\AOBSa{A_{\text{\eqref{OBS1}}}}
\def\bOBSa{b_{\text{\eqref{OBS1}}}}
\def\VOBSa{V_{\text{\eqref{OBS1}}}}
\def\dVOBSa{\dot{V}_{\text{\eqref{OBS1}}}}
\def\uOBSa{u_{\text{\eqref{OBS1}}}}
\def\deltaOBSa{\delta_{\text{\eqref{OBS1}}}}
\def\FOBSa{F_{\text{\eqref{OBS1}}}}
\def\rhoOBSa{\rho_{\text{\eqref{OBS1}}}}
\def\omOBSa{\omega_{\text{\eqref{OBS1}}}}
\def\xiOBSa{\xi_{\text{\eqref{OBS1}}}}
\def\epsOBSa{\varepsilon_{\text{\eqref{OBS1}}}}

\def\XOBSb{X_{\text{\eqref{OBS2}}}}
\def\dXOBSb{\dot{X}_{\text{\eqref{OBS2}}}}
\def\AOBSb{A_{\text{\eqref{OBS2}}}}
\def\bOBSb{b_{\text{\eqref{OBS2}}}}
\def\VOBSb{V_{\text{\eqref{OBS2}}}}
\def\dVOBSb{\dot{V}_{\text{\eqref{OBS2}}}}
\def\uOBSb{u_{\text{\eqref{OBS2}}}}
\def\deltaOBSb{\delta_{\text{\eqref{OBS2}}}}
\def\FOBSb{F_{\text{\eqref{OBS2}}}}
\def\rhoOBSb{\rho_{\text{\eqref{OBS2}}}}
\def\omOBSb{\omega_{\text{\eqref{OBS2}}}}
\def\xiOBSb{\xi_{\text{\eqref{OBS2}}}}
\def\epsOBSb{\varepsilon_{\text{\eqref{OBS2}}}}

\def\cR{{\cal R}}

\renewcommand{\t}{^{\mbox{\tiny\sf T}}}

\begin{abstract}
 The issue of state estimation is considered for an SIR-SI epidemiological model describing a vector-borne disease such as dengue fever, subject to seasonal variations.
 Assuming continuous measurement of the incidence rate (that is the number of new infectives in the host population per unit time), a class of interval observers with estimate-dependent gain is constructed, and asymptotic error bounds are provided.
The synthesis method is based on the search for a common linear Lyapunov function for monotone systems that represent the evolution of the estimation errors.
\end{abstract}

\section{ Introduction}
\label{se1}

 Vectors are living organisms that can transmit infectious diseases between humans or from animals to humans.
 Many of them are bloodsucking insects, which ingest disease-producing microorganisms during a blood meal from an infected host and inject it into a new host during a subsequent blood meal.
Vector-borne diseases account for more than 17\% of all infectious diseases, causing more than 1 million deaths annually.
As an example, more than 2.5 billion people in over 100 countries are at risk of contracting dengue, a vector-borne disease transmitted by some mosquito species of the genus {\em Aedes} \cite{Organization:2016aa}.
Several of the arboviroses transmitted by the latter (Zika fever, chikungunya, dengue) have no satisfying vaccine or curative treatment so far, and prevention of the epidemics is a key to the control policy.
In particular, the knowledge of the stock of susceptible individuals constitutes, among others, an important information to evaluate the probability of occurrence of such events. 


The dynamics of the dengue transmission is a quite complex subject, due to the role of the cross-reactive antibodies for the four different dengue serotypes \cite{Andraud:2017aa}.
We disregard here this multiserotype aspect, and focus on a basic model of the evolution of a vector-borne disease \cite{Dietz:1975aa,Esteva:1998aa}.
The latter is an SIR-SI \cite{Keeling:2008aa} compartmental model with vital dynamics (see equation \eqref{eq0} below).
It describes the evolution of the relative proportions of three classes in the host population: the susceptibles $S$ capable of contracting the disease and becoming infective, the infectives $I$ capable of transmitting the disease to susceptibles vectors, and the recovered $R$ permanently immune after healing; and only the two corresponding classes $S$ and $I$ in the vector population.
The denomination ``SIR-SI" of the model recalls this transmission mechanism.
The host incidence rate is assumed to be measured.

We are interested in this paper in estimating the repartition of the host and vector populations in time-varying, uncertain, conditions.
Our contribution is a class of {\em interval observers} \cite{Gouze:2000aa} providing lower and upper estimates for the state components.
Recent progress has been made in the design of interval observers for nonlinear systems, see  \cite{Raissi:2012aa,Efimov:2013aa} and a recent survey in \cite{Efimov:2016aa}.
The most recent methods usually involve linearization of the system through the use of the {\em Observer Canonical Form,} and then synthesis of interval observer for the corresponding Linear Parameter-Varying system.
However the state of the model considered here has dimension 5 and the scalar output is nonlinear with respect to the latter: in practice the complexity of the computations renders this approach intractable here.

An interval observer is therefore designed directly (in the spirit e.g.\ of \cite{Moisan:2009aa} and related contributions).
We characterize its behavior
via the search for {\em common linear Lyapunov functions} \cite{Mason:2007aa} and by use of the theory of {\em monotone systems} \cite{Hirsch:1988aa,Smith:1996aa}.
An important feature is that the observer gains are chosen in order to maximize the convergence speed of the Lyapunov functions towards zero.
This permits to exploit the epidemic bursts to speed up the state estimation, through estimate-dependent gain scheduling.
Error estimates are provided that  show, in particular, fast convergence towards the true values in absence of uncertainties on the transmission rates.
Notice that related ideas have been applied to an SIR infection model \cite{Bliman:2017aa}.
Similarly to the situation presented in this reference, the epidemiological system studied here is {\em unobservable} in absence of infected hosts.

The paper is organized as follows.
The model is introduced and commented in Section \ref{se15}.
The main assumptions are presented in Section \ref{se2}, where some qualitative results are also provided.
The considered class of observers is given in Section \ref{se3}, together with some {\em a priori} estimates, and is proved to constitute a class of interval observers.
The main result (Theorem \ref{MainTheorem}) is provided in Section \ref{se4}, where the asymptotic error corresponding to adequate gain choice is quantified.
A numerical illustration is presented in Section \ref{se55}.
Some concluding remarks are given in Section \ref{se5}.
Several proofs have been gathered in the Appendix.

\section{Statement of the problem}
\label{se15}

Using index $h$ for the host population and $v$ for the vectors, the SIR-SI model writes
\begin{subequations}
\label{eq0}
\begin{align}
\dot{S_h} &= \mu_h - \beta_{vh}(t) S_h I_v - \mu_h S_h\\
\dot{I_h} &= \beta_{vh}(t)S_h I_v - (\mu_h + \gamma) I_h\\
\dot{R_h} &= \gamma I_h - \mu_h R_h\\
\dot{S_v} &= \mu_v - \beta_{hv}(t) S_v I_h - \mu_v S_v\\
\dot{I_v} &= \beta_{hv}(t) S_v I_h - \mu_v I_v \\
y(t) & = \beta_{vh}(t) S_h(t)  I_v(t)
\end{align}
\end{subequations}
The parameters $\mu_h$, $\mu_v$ and $\gamma$ represent, respectively, the death rates for hosts and vectors, and the host recovery rate.
They are assumed to be constant and, as is e.g.\ the case for dengue,  the disease is supposed to induce no significant supplementary mortality in the infected individuals of both populations.
The latter are assumed to be stationary, hence $\mu_h$, $\mu_v$ also represent birth rates.
Notice that there is no recovery for the vectors, whose life duration is short compared to the disease dynamics.

The parameters $\beta_{vh}(t)$ and $\beta_{hv}(t)$ visible in the source terms for the evolution of the infected are, respectively, the {\em transmission rates} between infective vectors and susceptible hosts, and between infective hosts and susceptible vectors.
The conditions of transmission of the infection display significant seasonal variability (that may relate to changes in the climatic conditions, in the host population behavior\dots) and we therefore consider time-varying, uncertain, transmission rates.
The only measurement of the state components typically accessible to Public Health Services is the {\em host incidence}, i.e.\ the number of new infected hosts per unit time given by $y(t)=\beta_{vh}(t) S_h(t)  I_v(t)$.
This nonlinear expression of the state variable is the {\em measured output} of the model.

By construction, one has $\dot S_h+\dot I_h +\dot R_h = \dot S_v+\dot I_v = 0$, as the proportions verify $S_h+ I_h + R_h =  S_v+ I_v \equiv 1$.
This allows to remove one compartment of each population from the model, obtaining the following simplified  system:
\begin{subequations}
\label{SII}
\begin{align}
\label{dotSh} \dot{S_h} &= \mu_h - \beta_{vh}(t) S_h I_v - \mu_h S_h,\\
\label{dotIh} \dot{I_h} &= \beta_{vh}(t) S_h I_v - (\mu_h + \gamma) I_h,\\
\label{dotIv} \dot{I_v} &= \beta_{hv}(t) (1 - I_v) I_h - \mu_v I_v,\\
y(t) & =\beta_{vh}(t) S_h(t)  I_v(t)
\end{align}
\end{subequations}

When the transmission rates $\beta_{vh}, \beta_{hv}$ are constant, the evolution of the solutions of system \eqref{SII} depends closely upon the {\em basic reproduction ratio} $\cR_0 := \frac{\beta_{vh}\beta_{hv}}{(\mu_h+\gamma)\mu_v}$.
The disease-free equilibrium, defined by $S_h=1$, $I_h=I_v=0$ (and therefore $R_h=0$, $S_v=1$), always exists.
When $\cR_0<1$, it is the only equilibrium and it is globally asymptotically stable.
It becomes unstable when $\cR_0>1$, and an asymptotically stable endemic equilibrium then appears \cite{Esteva:1998aa}.
In practice, when the parameters are time-varying, complicated dynamics may occur.
This is especially the case due to seasonal changes of the climatic conditions, alternating between periods favorable and unfavorable to the occurrence of epidemic bursts.

{\em Our aim in this paper is to estimate at each moment the amount of the three involved populations $S_h(t), I_h(t), I_v(t)$, based on the measured incidence $y(t)$}.
An observer for this system has been proposed in the case of constant, known, parameters  \cite{Tami:2014aa} (assuming the availability of measures of $S_h$ and $I_h$), and this is up to our knowledge the unique contribution on the subject.
The challenging nature of the problem comes from the fact that when $I_v\equiv 0$, then $y\equiv 0$ and system \eqref{SII} is therefore unobservable.
\comment{
We introduce noise in measurement and unmodeled dynamics, and system \eqref{SII} that will be studied in the sequel is finally as follows:
\begin{subequations}
\label{SII}
\begin{align}
\label{dotSh} \dot{S_h} &= \mu_h - \beta_{vh}(t) S_h I_v - \mu_h S_h +w_1(t),\\
\label{dotIh} \dot{I_h} &= \beta_{vh}(t) S_h I_v - (\mu_h + \gamma) I_h +w_2(t),\\
\label{dotIv} \dot{I_v} &= \beta_{hv}(t) (1 - I_v) I_h - \mu_v I_v +w_3(t),\\
y(t) & = \beta_{vh}(t) S_h(t) I_v(t) +w_4(t)
\end{align}
\end{subequations}
}

\section{Model assumptions and properties}
\label{se2}

The host mortality rate (corresponding typically to several tens of years for humans) is assumed very small compared to both the recovery and the vector mortality rates (corresponding typically to some weeks).
Our first hypothesis is therefore:
\begin{equation}
\label{eq26}
\mu_h \ll \gamma,\quad  \mu_h \ll  \mu_v.
\end{equation}

As said previously, the transmission rates $\beta_{vh}$ and $\beta_{hv}$ are supposed to be subject to seasonal variations.
We assume moreover that their exact value is not known, but that they are bounded by lower and upper estimates $\beta_{vh}^-(t),\beta_{vh}^+(t)$ and $\beta_{hv}^-(t),\beta_{hv}^+(t)$, {\em available in real-time}: for any $t\geq 0$,
\begin{equation}
\beta_{vh}^-(t) \leq \beta_{vh}(t) \leq \beta_{vh}^+(t),\qquad
\beta_{hv}^-(t) \leq \beta_{hv}(t) \leq \beta_{hv}^+(t).
\end{equation}
\comment{
On the other hand, the functions $w_i$, $i=1,\dots, 4$, are assumed measured functions, and one assumes that exist integrable functions $w_i^\pm$, $i=1,\dots, 4$ such that for a.e.\ $t\geq 0$,
\begin{equation}
w_i^-(t) \leq w_i(t) \leq w_i^+(t),\qquad i=1,\dots, 4
\end{equation}
}


Our first result shows that the solutions of system \eqref{SII} respect the expected bounds.

\begin{lemma}[Properties of model \eqref{SII}]
\label{le1}
The solutions of system \eqref{SII} are such that: if $S_h(t)\geq 0, I_h(t)\geq 0, I_v(t)\geq 0,$ $S_h(t)+I_h(t)\leq 1, I_v(t)\leq 1$ for $t=0$,  then the same properties hold for any $t\geq 0$.
Similarly for strict inequalities.
\hfill$\square$
\end{lemma}

See in Appendix A a proof of Lemma \ref{le1}.

\section{Observers: definition and monotonicity properties}
\label{se3}

We  now introduce the two following systems, and show in the sequel that, under appropriate conditions, they constitute interval observers for system \eqref{SII}:
\begin{subequations}\label{OBS1}
\begin{align}
\label{OBS1a}
\hspace{-.2cm}
\dot{S}_h^- &= \mu_h (1 - {S}_h^-) - y + k_S^-(t) (y - {\beta}_{vh}^+(t) {S}_h^-{I}_v^+) \\
\label{OBS1b}
\hspace{-.2cm}
\dot{I}_h^+ &= y - (\mu_h + \gamma){I}_h^+ 
\\
\label{OBS1c}
\hspace{-.2cm}
\dot{I}_v^+ &= \beta_{hv}^+(t) (1 - {I}_v^+) {I}_h^+ - \mu_v {I}_v^+ + k_{v}^+(t) (y - {\beta}_{vh}^-(t) {S}_h^-{I}_v^+)
\end{align}
\end{subequations}
\vspace{-.5cm}
\begin{subequations}\label{OBS2}
\begin{align}
\label{OBS2a}
\hspace{-.2cm}
\dot{S}_h^+ &= \mu_h (1 - {S}_h^+) - y +  k_S^+(t) (y - \beta_{vh}^-(t) {S}_h^+ {I}_v^-) \\
\label{OBS2b}
\hspace{-.2cm}
\dot{I}_h^- &= y - (\mu_h + \gamma){I}_h^- 
 \\
\label{OBS2c}
\hspace{-.2cm}
\dot{I}_v^- &= \beta_{hv}^-(t) (1 - {I}_v^-) {I}_h^- - \mu_v {I}_v^- + k_{v}^-(t) (y - {\beta}_{vh}^+(t) {S}_h^+{I}_v^-)
\end{align}
\end{subequations}
As can be seen, output injection (the output is $y$) is used in this synthesis.
The time-varying gains $k_S^\pm(t), k_v^\pm(t)$ in \eqref{OBS1a}, \eqref{OBS1c}, \eqref{OBS2a}, \eqref{OBS2c} are yet to be defined.

\begin{lemma}[Nonnegativity of the estimates]
\label{le2}
  Suppose that for some  $\varepsilon_1$, $\varepsilon_2>0,$ the gains $k_S^\pm(t), k_v^\pm(t)$ are chosen in such a way that, for any $t\geq 0$ :
  \begin{equation}\label{HypkS}
    \mu_h+ (k_S^\pm(t)-1)y(t) \geq \varepsilon_1, \quad \text{ whenever } S_h^\pm(t)\leq \varepsilon_2
  \end{equation}

 Then the solutions of system \eqref{OBS1} are such that:  if $S_h^-(t), I_h^+(t), I_v^+(t)\geq 0$ for $t=0$, then the same remains true for any $t>0$.
 Analogously, the solutions of system \eqref{OBS2} are such that:  if $S_h^+(t), I_h^-(t), I_v^-(t)\geq 0$ for $t=0$, then the same remains true for any $t>0$.
\hfill$\square$
\end{lemma}

\begin{proof}
Observe that, whenever $S_h^-$ is close to zero, one has $\dot{S}_h^- \sim \mu_h + (k_S^-(t)-1)y \geq \varepsilon_1 >0$,  due to \eqref{HypkS}.
One also has that $\dot{I}_ h^+$ is nonnegative whenever $I_h^+$ is in a neighborhood of $0.$ The same holds for $I_v^+$, and this concludes the proof of Lemma \ref{le2} for system \eqref{OBS1}.
The proof for \eqref{OBS2} is analogous.
\end{proof}

Condition \eqref{HypkS} imposes special care in the choice of the gains when the estimates $S_h^\pm(t)$ come close to zero.

The next result shows that, under sufficient conditions on the gain, equations \eqref{OBS1} and \eqref{OBS2} provide an interval observer for \eqref{SII}, this is, they provide upper and lower estimates of the three state components.

\begin{theorem}[Ordering property of the estimates]
\label{th2}
Assume that  the gains $k_S^\pm(t), k_v^\pm(t)$ are chosen in such a way that, for any $t\geq 0$, $k_S^\pm(t)\geq 0, k_v^\pm(t)\geq 0$ and that condition \eqref{HypkS} is verified. 
Suppose that  the solutions of \eqref{OBS1}, \eqref{OBS2} are such that
\begin{subequations}
\label{eq8}
\begin{gather}
0\leq S_h^-(t) \leq S_h(t) \leq S_h^+(t)\\
0\leq I_h^-(t) \leq I_h(t) \leq I_h^+(t)\\
0\leq I_v^-(t) \leq I_v(t) \leq I_v^+(t)
\end{gather}
\end{subequations}
for $t=0$.
Then the same holds for any $t\geq 0$.
\hfill$\square$ 
\end{theorem}

To demonstrate Theorem \ref{th2}, an instrumental decomposition result is now given, whose proof can be found in Appendix B.
Lemma \ref{le3} indicates in particular that the errors obey some linear {\em positive} systems: based on this remark, the proof of Theorem \ref{th2} is straightforward.

\begin{lemma}[Dynamics of the observer errors]
\label{le3}
The observer errors attached to \eqref{OBS1}, \eqref{OBS2}, defined by
\begin{subequations}
\label{eq22}
\begin{gather}
\label{eq22a}
X_{\text{\eqref{OBS1}}}
:=
\begin{pmatrix} 
e_S^- \\ e_h^+ \\ e_v^+ 
\end{pmatrix}
:=
\begin{pmatrix}
S_h - S_h^-\\
 I_h^+ - I_h \\
 I_v^+ - I_v 
\end{pmatrix}\\
\label{eq22b}
\XOBSb:=
\begin{pmatrix} 
e_S^+ \\ e_h^- \\ e_v^- 
\end{pmatrix}
:=
\begin{pmatrix}
S_h^+ - S_h\\
 I_h - I_h^- \\
 I_v - I_v^- 
\end{pmatrix}
\end{gather}
\end{subequations}
fulfill, for any $t\geq 0$, the equations
\begin{subequations}
\label{eq15}
\begin{gather}
\label{eq15a}
\dXOBSa (t) = \AOBSa(t)\XOBSa(t) + \bOBSa(t)\\
\label{eq15b}
\dXOBSb(t) =  \AOBSb(t) \XOBSb(t)+  \bOBSb(t)
\end{gather}
\end{subequations}
where $\AOBSa(t), \AOBSb(t)$ and $\bOBSa(t), \bOBSb(t)$ are defined in \eqref{eq16}.
\begin{figure*}
\begin{subequations}
\label{eq16}
\begin{gather}
\label{eq16a}
\AOBSa(t)
= 
\begin{pmatrix}
-\mu_h - k_S^- {\beta}_{vh}^+ {I}_v^+ & 0 & k_S^- {\beta_{vh}}^+ {S}_h  \\
 0 & -(\mu_h + \gamma) & 0 \\
k_v^+ {\beta_{vh}}^- {I}_v^+  & {\beta}_{hv}^+ (1 - I_v) & -  k_v^+ {\beta_{vh}}^- S_h - \beta_{hv}^+ {I}_h^+- \mu_v \\
\end{pmatrix}\hspace{-.12cm} , \hspace{.01cm}
\bOBSa(t)=
\begin{pmatrix}
k_S^- (\beta_{vh}^+ - {\beta_{vh}}) S_h I_v \\
0 \\
k_v^+  (\beta_{vh} - {\beta_{vh}}^-) S_h I_v+ (\beta_{hv}^+- {\beta}_{hv}) I_h (1-I_v)
\end{pmatrix}\\
\label{eq16c}
\AOBSb(t)
= 
\begin{pmatrix}
-\mu_h - k_S^+ {\beta}_{vh}^- {I}_v^- & 0 & k_S^+ {\beta_{vh}}^- {S}_h  \\
 0 & -(\mu_h + \gamma) & 0 \\
k_v^- {\beta_{vh}}^+ {I}_v^-  & {\beta}_{hv}^- (1 - I_v) & -  k_v^- {\beta_{vh}}^+ S_h - \beta_{hv}^- {I}_h^-- \mu_v \\
\end{pmatrix}\hspace{-.12cm} , \hspace{.01cm}
\bOBSb(t)=
\begin{pmatrix}
k_S^+ (\beta_{vh} - {\beta_{vh}}^-) S_h I_v \\
0 \\
k_v^-  (\beta_{vh}^+ - {\beta_{vh}}) S_h I_v  + (\beta_{hv}- {\beta}_{hv}^-) I_h (1-I_v)
\end{pmatrix}
\end{gather}
\end{subequations}
\hrule
\end{figure*}

Moreover, for any $t\geq 0$, the matrices $\AOBSa(t), \AOBSb(t)$ are Metzler\footnote{A square matrix is called a {\em Metzler matrix} \cite{Farina:2011aa} or an {\em essentially nonnegative matrix} \cite{Berman:1994aa} if all its off-diagonal components are nonnegative.}, and the vectors $\bOBSa(t), \bOBSb(t)$ are nonnegative, and null in the absence of uncertainties (that is when $\beta_{vh}^-(t) = \beta_{vh}(t) = \beta_{vh}^+(t)$ and $\beta_{hv}^-(t) = \beta_{hv}(t) = \beta_{hv}^+(t)$).
\hfill$\square$
\end{lemma}

Observe that the matrices $\AOBSa(t), \AOBSb(t)$ and the vectors $\bOBSa(t), \bOBSb(t)$ defined in Lemma \ref{le3} depend upon time through the values of the transmission rates and their upper and lower estimates, but also through components of the initial system \eqref{SII} and of the observers \eqref{OBS1} and \eqref{OBS2}.
For brevity this time-dependence is not explicitly shown in \eqref{eq16}.

As a last remark, notice that easy computations (not reproduced here) establish that the verification of the Metzler property in Lemma \ref{le3} implies that no gain $k_h^\pm(t)$ should be introduced in \eqref{OBS1b}, \eqref{OBS2b}.


\if{

\section{Observers: error estimates}

Let $S_h^+$, $I_h^+,$ $I_v^+$ be upper observers and $S_h^-$, $I_h^-,$  $I_v^-$ lower observers for $S_h$, $I_h$ and $I_v,$ respectively.
We consider the two independent differential systems for the error vectors given by
\begin{equation}
X^+:=
\begin{pmatrix} 
e_S^- \\ e_h^+ \\ e_v^+ 
\end{pmatrix}
:=
\begin{pmatrix}
S_h - S_h^-\\
 I_h^+ - I_h \\
 I_v^+ - I_v 
\end{pmatrix}
;\quad 
X^-:=
\begin{pmatrix} 
e_S^+ \\ e_h^- \\ e_v^- 
\end{pmatrix}
:=
\begin{pmatrix}
S_h^+ - S_h\\
 I_h - I_h^- \\
 I_v - I_v^- 
\end{pmatrix}
\end{equation}
In order to guarantee the order between the variables and its estimatives, we want that the errors mantain theirs signs, this is, if they start with positive (negative) values they should remain at positive (negative) values. For this it is sufficient that the matrix $A$ in \eqref{ErrorSystem} has nonnegative off-diagonal terms and that the vector $b$ of the right-hand side of the equation \eqref{ErrorSystem} has nonnegative entries. In order to guarantee the nonnegativity of the off-diagonal terms for $A$ it is sufficient to set $\hat k_h\equiv 0.$

For the system of variable $X^+:=(e_S^-,e_h^+,e_v^+)^\top$ we get the dynamics

}\fi

\section{Observers: convergence properties}
\label{se4}

We now consider the issue of ensuring fast convergence of the estimates towards zero in the absence of uncertainties.
As may be noticed, the dynamics of the (nonnegative) errors $e_h^\pm = |I_h-I_h^\pm|$
(see the 2nd lines of \eqref{eq15a}, \eqref{eq15b}) is not modified by the gain choice: the estimates of the proportion of infective hosts $I_h$ converge at a constant rate $\mu_h+\gamma$, which essentially depends upon the recovery rate $\gamma$, as $\mu_h\ll\gamma$.

Observe that removing the gains $k_S^\pm(t)$ and $k_v^\pm(t)$ yields converging estimates (see the matrices $\AOBSa, \AOBSb$ in \eqref{eq16a}, \eqref{eq16c}), with a convergence rate at most equal to the host birth/death rate $\mu_h$.
But since $\mu_h$ is very small, it is unsatisfactory for practical matters  to settle for such a slow speed of convergence.
We provide in Theorem \ref{MainTheorem} a way to cope with this issue.
See Appendix C for a complete proof.

\begin{theorem}[Convergence property of the estimates]
\label{MainTheorem}
Assume that the initial conditions verify \eqref{eq8} given in Theorem {\rm\ref{th2}}.
Assume that, for {\em fixed positive} scalar numbers $ \omOBSa, \omOBSb,\epsOBSa,\epsOBSb,$ the gains $k_S^\pm(t), k_v^\pm(t)$ fulfill the assumptions of Theorem {\rm\ref{th2}} as well as the conditions:
\begin{subequations}
 \label{ell}
 \begin{gather}
 \label{ella}
 k_S^-(t)\beta_{vh}^+(t)  -\omOBSa k_v^+(t)\beta_{vh}^-(t) = \xiOBSa(t),\\
 \label{ellb}
 k_S^+(t)\beta_{vh}^-(t)  -\omOBSb k_v^-(t)\beta_{vh}^+(t) = \xiOBSb(t),
\end{gather}
\end{subequations}
where\footnote{If $I_v^\pm(t)=0$ in \eqref{xia}/\eqref{xib}, then the $\min$ selects the other term.}
\begin{subequations}
\label{xi}
\begin{gather}
\label{xia}
 \xiOBSa(t):=
  \min\left\{ \frac{\gamma-{\epsOBSa}}{I_v^+(t)}; \frac{\mu_v-\mu_h+\beta_{hv}^+(t)I_h^+(t)}{S_h^+(t)/ \omOBSa+I_v^+(t)}\right\},\\
\label{xib}
 \xiOBSb(t):= \min\left\{\frac{ \gamma-\epsOBSb}{I_v^-(t)};\frac{\mu_v-\mu_h+\beta_{hv}^-(t)I_h^-(t)}{S_h^+(t)/\omOBSb+I_v^-(t)}\right\}.
\end{gather}
\end{subequations}
Then, along any trajectories of \eqref{OBS1}, \eqref{OBS2} one has, for all $t\geq 0$,
\begin{subequations}
\label{Vestimates}
\begin{gather}
\label{Vestimatesa}
0\leq \VOBSa(t) \leq e^{ -\int_0^t\deltaOBSa(s) ds}\ \VOBSa(0) + \int_0^t e^{-\int_s^t\deltaOBSa(\tau) d\tau} \FOBSa(s) ds,\\
\label{Vestimatesb}
 0\leq \VOBSb(t) \leq e^{ -\int_0^t\deltaOBSb(s) ds} \VOBSb(0) + \int_0^t e^{-\int_s^t\deltaOBSb(\tau) d\tau} \FOBSb(s) ds,
 \end{gather}
 \end{subequations}
where 
\begin{subequations}
\label{Vdef}
\begin{gather}
\label{Vdefa}
 \VOBSa = (S_h - S_h^-) + \rhoOBSa (I_h^+ - I_h) + \omOBSa (I_v^+ - I_v),\\
\label{Vdefb}
 \VOBSb = (S_h^+ - S_h) + \rhoOBSb (I_h - I_h^-) + \omOBSb (I_v - I_v^-),
\end{gather}
\end{subequations}
are positive definite functions, 
\begin{subequations}
\label{rhoa}
\begin{gather}
\rhoOBSa:= {\omOBSa} \sup_{t\geq 0} \frac{\beta_{hv}^+(t)(1-I_v^-(t))}{\epsOBSa},\\
\label{rhob}
\rhoOBSb:= {\omOBSb} \sup_{t\geq 0} \frac{\beta_{hv}^-(t)(1-I_v^-(t))}{\epsOBSb}.
\end{gather}
\end{subequations}
\begin{subequations}
\label{delta}
\begin{gather}
\label{deltaa}
\deltaOBSa(t) := \mu_h+ \xiOBSa(t)I^+_v(t),\\
\label{deltab}
\deltaOBSb(t) := \mu_h+ \xiOBSb(t)I^-_v(t),
\end{gather}
\end{subequations}
and
\begin{subequations}
 \label{Fdef}
 \begin{eqnarray}
  \FOBSa
 & :=  &
 \nonumber
 \omOBSa \Big( k_v^+(\beta_{vh}^+-\beta_{vh}^-)S_hI_v+(\beta_{hv}^+-\beta_{hv}^-)I_h(1-I_v)\Big)\\
 & &
 \label{Fdefa}
+ k_S^-(\beta_{vh}^+-\beta_{vh}^-)S_hI_v\\
 \FOBSb
 & := &
 \nonumber
 \omOBSb \Big( k_v^-(\beta_{vh}^+-\beta_{vh}^-)S_hI_v+(\beta_{hv}^+-\beta_{hv}^-)I_h(1-I_v)\Big)\\
 & &
 \label{Fdefb}
 + k_S^+(\beta_{vh}^+-\beta_{vh}^-)S_hI_v
  \end{eqnarray}
\end{subequations}
\hfill$\square$
\end{theorem}

It is easy to verify that it is always possible to find gain values that fulfill (pointwise) \eqref{ell}-\eqref{xi}.
When all assumptions of Theorem \ref{MainTheorem} are satisfied,
inequalities \eqref{Vestimates} provide guaranteed bounds on the error estimates.
In the absence of uncertainties, $\beta_{vh}^+\equiv\beta_{vh}^-$, $\beta_{hv}^+\equiv\beta_{hv}^-$, and the functions in \eqref{Fdef} are identically null (see Lemma \ref{le3}):
only the first terms remain in the right-hand sides of \eqref{Vestimates}, and the errors converge towards zero exponentially.

The speed of convergence is dictated by the values of the {\em positive} (due to Lemma \ref{le2}) instantaneous convergence rates $\deltaOBSa(t), \deltaOBSb(t)$ given in formulas \eqref{delta}.
An important point is that convergence may occur with quite a slow pace: when the estimates $I_v^\pm$ on the infectives $I_v$ are small, it occurs at the natural rate $\mu_h$ of system \eqref{SII}.
On the contrary, its speed increases with the value of $I_v^\pm$: the observers take advantage of the epidemic bursts to provide tighter estimates more rapidly.
By construction $\xiOBSa(t), \xiOBSb(t) \leq \gamma$, so $\deltaOBSa(t), \deltaOBSb(t)$ are at most equal to $\mu_h+\gamma$: the interest of the gain choice made in \eqref{ell}-\eqref{xi} is to strive for this performance.

\begin{remark}[On noise and unmodeled dynamics]
Noise in measurement and unmodeled dynamics appearing additively in the right-hand sides of system \eqref{SII} can be handled without major difficulty (this is omitted here for sake of space).
When the latter are upper and lower bounded by some known signals, introducing {\em ad hoc} linear expressions of the latter in the right-hand sides of \eqref{OBS1}, \eqref{OBS2} and exploiting the linearity of the Lyapunov functions allows to obtain formulas analogous to \eqref{Vestimates}, with some additional terms in the expression of the functions $\FOBSa$, $\FOBSb$.
\hfill$\square$
\end{remark}

\section{Numerical simulations}
\label{se55}

The model parameters used for this test are listed in Table \ref{parametros}.
In order to represent seasonal variations, the transmission rates were taken periodic,  $\beta_{vh}(t) = \beta_{vh,0}(1+0.4\cos(2\pi t))$, $\beta_{hv}(t) = \beta_{hv,0}(1+0.4\cos(2\pi t))$, and the uncertainties were defined by $\beta_{vh}^\pm(t) = (1\pm 0.1)\beta_{vh}(t)$, $\beta_{hv}^\pm(t) = (1\pm 0.1)\beta_{hv}(t)$.
The gains were chosen, in accordance with Theorem \ref{MainTheorem}, as
\begin{subequations}
\begin{gather}
 k_S^-(t) = \frac{\xiOBSa(t)}{\beta_{vh}^+(t)}
\text{ if } S_h^-(t) \geq \varepsilon_2,\\
k_S^-(t) = \max\left\{
\frac{\xiOBSa(t)}{\beta_{vh}^+(t)};
1-\frac{\mu_h-\varepsilon_1}{y(t)}
\right\} \text{ otherwise},\\
 k_v^+(t)= \frac{k_S^-(t) \beta_{vh}^+(t)-\xiOBSa(t)}{\omOBSa\beta_{vh}^-(t)}\\
k_S^+(t) = \frac{\xiOBSb(t)}{\beta_{vh}^-(t)},\qquad
 k_v^-(t)= 0,
\end{gather}
\end{subequations}
with $\omOBSa=\omOBSb=10^5$, $\epsOBSa=\epsOBSb=10^{-4}$, $\varepsilon_1=\varepsilon_2=10^{-5}$.
Last, the initial conditions were fixed at 
$S_h(0)=0.2, S_h^-(0)=0.1, S_h^+(0)=0.8$,
$I_h(0)=0, I_h^-(0)=0, I_h^+(0)=0.01$,
$I_v(0)=0.005, I_v^-(0)=0, I_v^+(0)=0.01$.


Figure \ref{figIv} represents respectively the corresponding evolution of $S_h, I_h, I_v$ (in blue) and their lower (in red) and upper (in green) estimates with respect to time (expressed in {\em years}).
Fast convergence of $I_h^\pm$ and $I_v^\pm$ towards $I_h$ and $I_v$ respectively is apparent, with quite small residual errors.
On the contrary, after a quick decrease, errors on the estimates of $S_h$ stay at relative value of about 20\%, due to the uncertainty on the transmission rates.

\begin{table}
\begin{center}
\begin{tabular}{|c|r|l|}
\hline
 \multicolumn{1}{|c|}{Parameter} & \multicolumn{1}{|c|}{Value} & \multicolumn{1}{|c|}{References}\\
\hline\hline
$\mu_h$  & $3.4\times 10^{-5}$ \ day$^{-1}$ & Massad et al. \cite{MassadCoutinho2010}  \\
$\beta_{vh,0}$ & $0.2102$  \ day$^{-1}$  & Massad et al. \cite{MassadCoutinho2010}  \\
$\gamma$ & $0.14$  \ day$^{-1}$  & Massad et al. \cite{MassadCoutinho2010}  \\
$\beta_{hv,0}$ & $0.1$  \ day$^{-1}$  & Massad et al. \cite{MassadCoutinho2010}  \\
$\mu_v$ & $0.025$  \ day$^{-1}$  & Massad et al. \cite{MassadCoutinho2010} \\
\hline
\end{tabular}
\caption{
List of parameter values 
}
\label{parametros}
\end{center}
\end{table}

\begin{figure}
    \centering
        \includegraphics[width=0.4\textwidth]{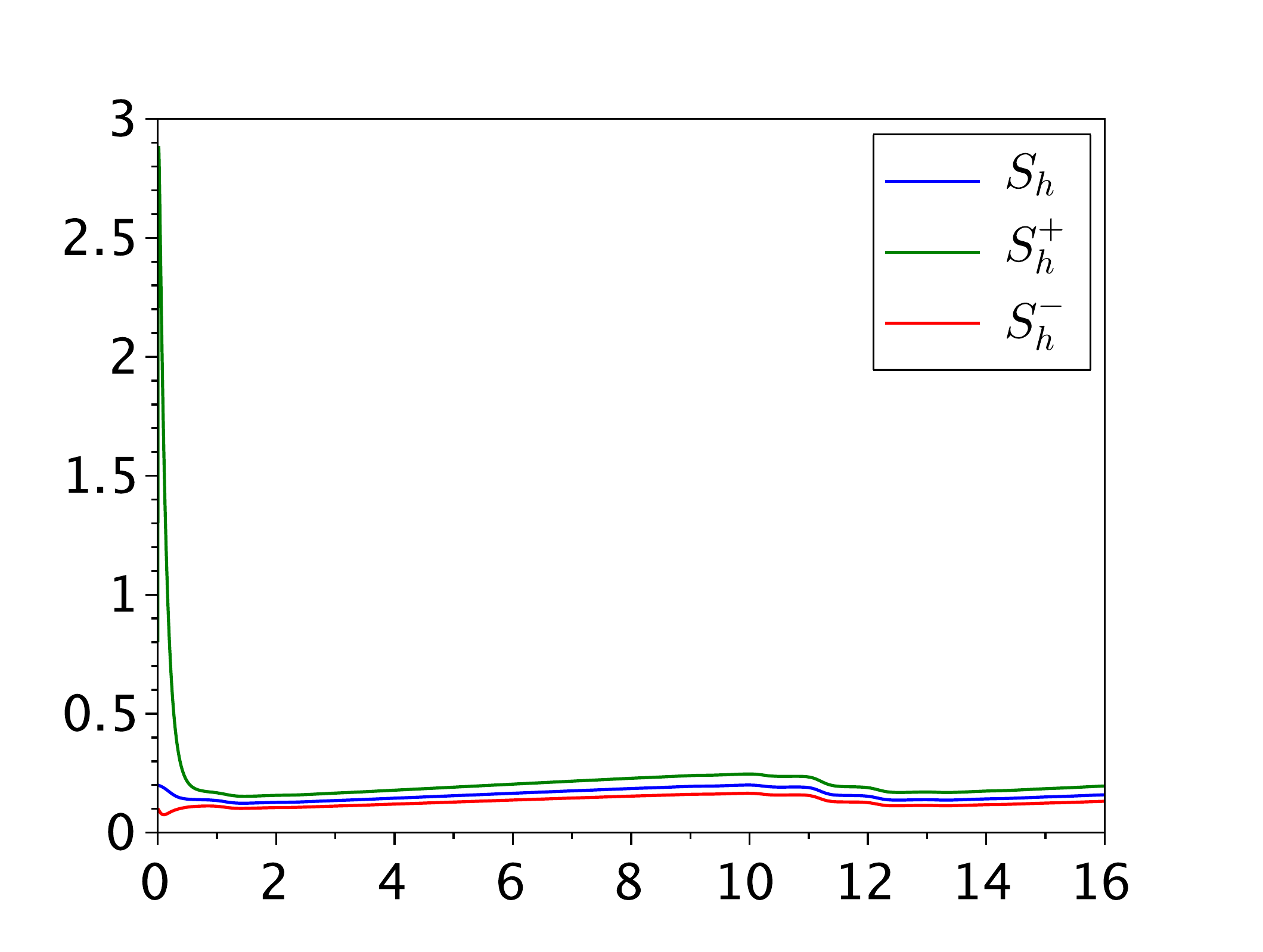}
         \includegraphics[width=0.4\textwidth]{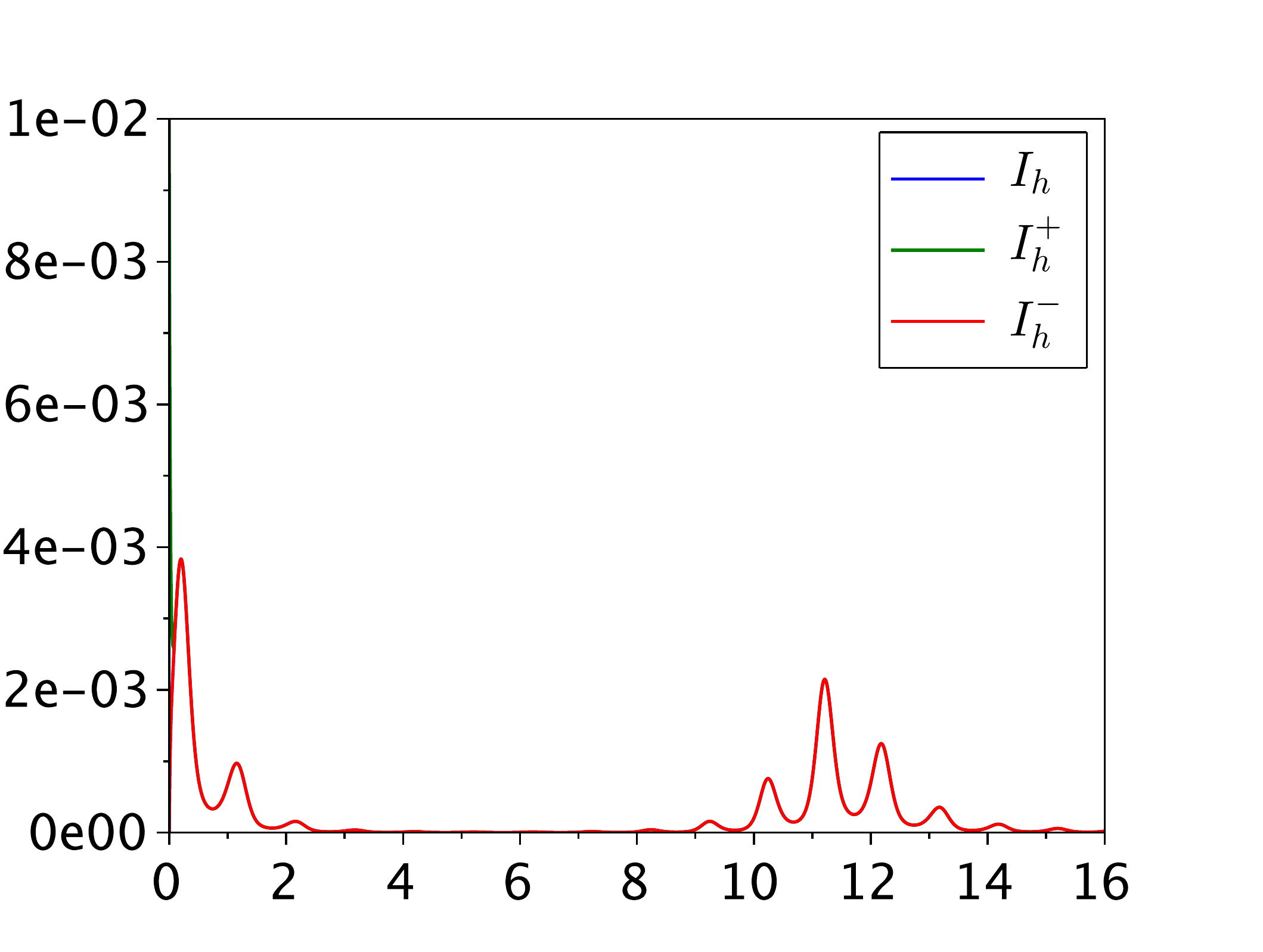}
        \includegraphics[width=0.4\textwidth]{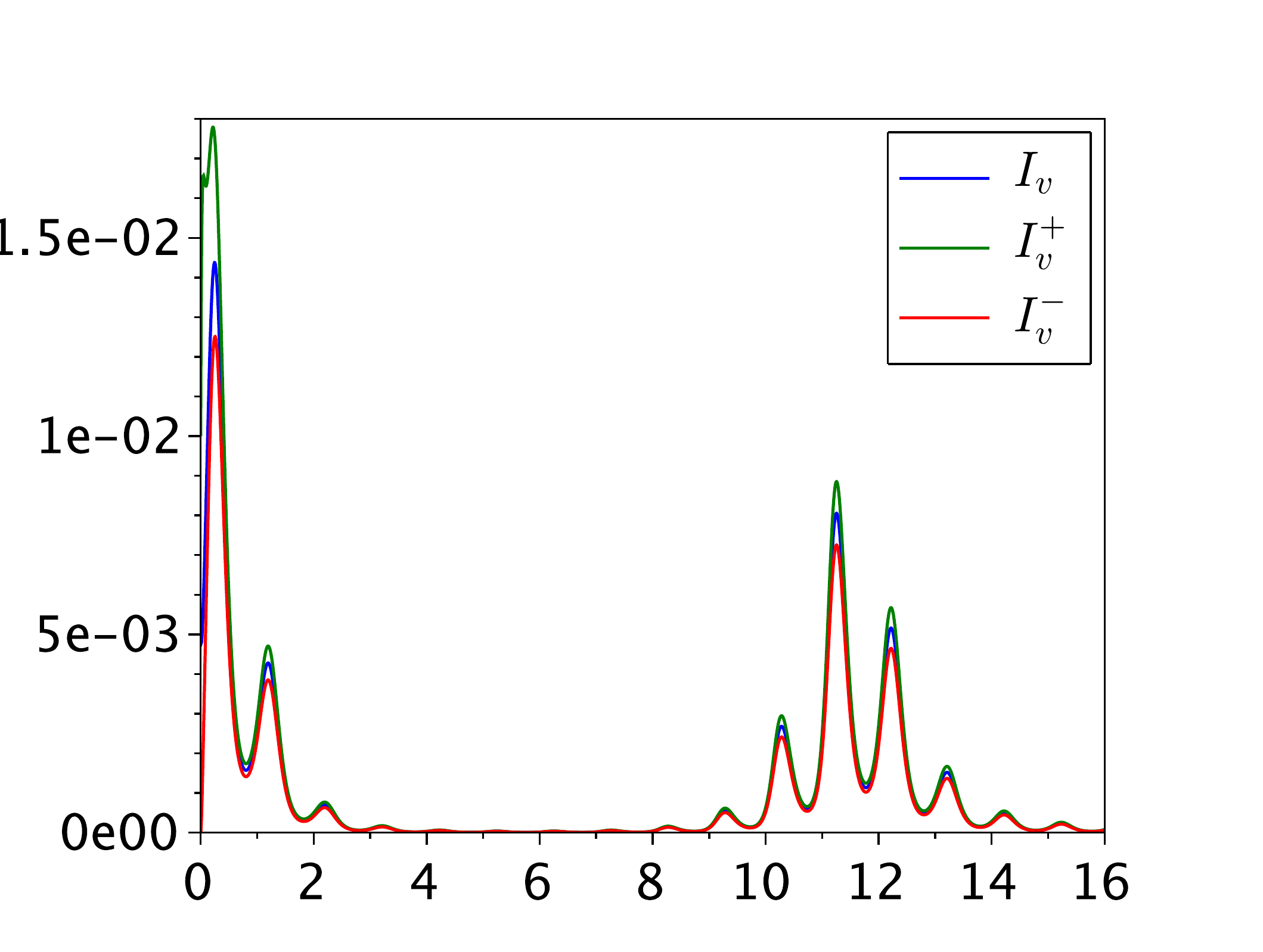}
        \caption{Proportions of susceptible hosts (top), infected hosts (middle) and infected vectors (bottom), and their estimates as functions of time (in {\em years}).
        The estimates of infected hosts become rapidly indistinguishable.}
        \label{figIv}
\end{figure}

\if{
\begin{figure}
    \centering
        \includegraphics[width=0.4\textwidth]{Tests/Ih.pdf}
        \caption{Infected human population and estimates}
        \label{figIh}
\end{figure}

\begin{figure}
    \centering
        \includegraphics[width=0.4\textwidth]{Tests/Iv.pdf}
        \caption{Infected vector population and estimates}
        \label{figIv}
\end{figure}
}\fi

\section{Conclusion}
\label{se5}

A class of interval observers has been provided for an SIR-SI model of vector-borne disease with time-varying uncertain transmission rates, assuming continuous measurement of the host incidence.
The system is unobservable in absence of infectives, and the proposed gain scheduling law speeds up the state estimation during the epidemic bursts.
Explicit bounds on the estimation errors have been given, which vanish asymptotically in absence of uncertainties.
Numerical test has been provided, demonstrating the expected behavior.
The analysis may be extended directly in presence of noise in measurement and unmodeled dynamics.

\appendix


\subsection{Proof of Lemma \ref{le1}}

Assume that $S_h(0)\geq 0,I_h(0)\geq 0,I_v(0)\geq 0.$
Observe that whenever $S_h=0,$ one has $\dot{S}_h=\mu_h>0.$ Hence, $S_h(t)\geq 0.$ Let us rewrite \eqref{dotIh}-\eqref{dotIv} in the more convenient form
\begin{equation}\label{IhIv}
 \begin{pmatrix}
  \dot I_h \\ \dot I_v
 \end{pmatrix}
 =
 \begin{pmatrix}
  -(\mu_h+\gamma) &  \beta_{vh}(t) S_h(t) \\
  \beta_{hv}(t) & -(\beta_{hv}(t)+\mu_v)
 \end{pmatrix}
 \begin{pmatrix}
  I_h \\ I_v
 \end{pmatrix}
\end{equation}
Since $S_h(t)$ is always nonnegative, system \eqref{IhIv} above turns out to be monotone \cite{Hirsch:1988aa,Smith:1996aa} and, consequently, $I_h(t)\geq0, I_v(t)\geq0$ whenever starting from nonnegative initial values.

On the other hand, if $S_h(t)+I_h(t)=1$ then $\frac{d}{dt}(S_h+I_h) = -\gamma I_h \leq 0$.
Also, whenever $I_v=1$, one has $\dot I_v = -\mu_v< 0.$ Hence, $S_h+I_h$ and $I_v$ remain under 1 if their initial values are located under 1.

In order to prove the same properties for the strict inequalities, first integrate \eqref{SII} to obtain, for any $t\geq 0$:
$S_h(t) = S_h(0)  e^{-\int_0^t (\mu_h+\beta_{vh}(s)I_v(s))  ds} + \mu_h \int_0^t e^{-\int_s^t (\mu_h+\beta_{vh}(\tau)I_v(\tau)) d\tau}ds$;
$I_h(t) = I_h(0) e^{-\int_0^t (\mu_h+\gamma)ds} +  \int_0^t e^{-(\mu_h+\gamma)(t-s)}\beta_{vh}(s)S_h(s)I_v(s)ds$ and
$I_v(t) = I_v(0) e^{-\int_0^t(\mu_v+\beta_{hv}(s)I_h(s))ds} + \int_0^t e^{-\int_s^t(\mu_v+\beta_{hv}(\tau)I_h(\tau))d\tau} \beta_{hv}(s)I_h(s)ds$.
From this we can easily deduce that all three variables remain strictly positive if starting from strictly positive initial values, using the previously proved nonnegativity of the three components.

Finally, note that,  for any $t\geq 0$:
$(1-S_h(t)-I_h(t)) = (1-S_h(0)-I_h(0))e^{-\int_0^t \mu_h ds} +  \int_0^t e^{-\int_s^t \mu_h  d\tau}\gamma I_h(s)ds$ and $(1-I_v(t)) = (1-I_v(0))e^{-\int_0^t \beta_{hv}(s)I_h(s) ds} +  \int_0^t e^{-\int_s^t \beta_{hv}(\tau)I_h(\tau)  d\tau} \mu_v I_v(s)ds$.
Thus  $S_h(0)+I_h(0)<1$ implies the same property for all $t>0$.
The same holds for $I_v.$
This ends the proof of Lemma \ref{le1}. 
\hfill$\blacksquare$

\subsection{Proof of Lemma \ref{le3}}

Let us first establish \eqref{eq15a}.
From \eqref{SII}, \eqref{OBS1} and the fact that $y=\beta_{vh}S_hI_v$, we get:
\begin{eqnarray}
\dot{e}_S^-
& = &
\nonumber
[\mu_h - \beta_{vh} S_h I_v - \mu_h S_h] \\
& &
\nonumber
- [\mu_h (1 - {S_h^-}) - y + k_S^- (y - \beta_{vh}^+ S_h^- I_v^+)] \\
& = &
\nonumber
\mu_h (S_h^- - S_h) - k_S^- (y - \beta_{vh}^+S_h^-I_v^+)\\
& = &
\label{doteS}
-\mu_h e_S^- -k_S^- (y - \beta_{vh}^+ S_h^- I_v^+)\\
\dot{e}_h^+
& = &
\nonumber
[y - (\mu_h + \gamma) I_h^+ ] - [\beta_{vh} S_h I_v - (\mu_h + \gamma) I_h] \\
& = &
\label{doteh}
- (\mu_h + \gamma)(I_h^+ - I_h) 
= -(\mu_h + \gamma) e_h ^+\\
\dot{e}_v^+
& = &
\nonumber
[{\beta_{hv}}^+ (1 - I_v^+) I_h^+ - \mu_v I_v^+ +  k_{v}^+ (y - {\beta_{vh}}^- S_h^- I_v^+)]\\
& &
\nonumber
- [\beta_{hv} (1 - I_v) I_h - \mu_v I_v]  \\
& = &
\nonumber
[{\beta}_{hv}^+ I_h^+ - {\beta}_{hv}^+ I_v^+ I_h^+ - \mu_v I_v^+ + k_v^+ (y - {\beta_{vh}}^- S_h^- I_v^+)]\\
& &
\nonumber
- [\beta_{hv} I_h - \beta_{hv} I_v I_h  - \mu_v I_v]
\end{eqnarray}
By adding and substracting $({\beta_{hv}}^+ I_h + {\beta}_{hv}^+ I_v I_h + {\beta}_{hv}^+ I_v I_h^+)$ on the right-hand side of the previous formula, we get
\begin{eqnarray}
\dot{e}_v^+
&= &
\nonumber
- (\beta_{hv} - \beta_{hv}^+) I_h - \beta_{hv}^+ (I_h - I_h^+) + (\beta_{hv} - \beta_{hv}^+) I_v I_h \\
& &
\nonumber
+ \beta_{hv}^+ I_v (I_h - I_h^+)
+ \beta_{hv}^+ (I_v - I_v^+) I_h^+ \\
& &
\nonumber
- \mu_v e_v^+ + k_v^+ (y - \beta_{vh}^- S_h^- I_v^+) \\
&= &
\nonumber
(\beta_{hv}^+ - \beta_{hv}) I_h + \beta_{hv}^+ e_h^+ - (\beta_{hv}^+ - \beta_{hv}) I_v I_h - \beta_{hv}^+ I_v e_h^+ \\
& &
\label{dotev2}
- \beta_{hv}^+ e_v^+ I_h^+ - \mu_v e_v^+ + k_v^+ (y - \beta_{vh}^- S_h^- I_v^+)
\end{eqnarray}

We now treat the terms $y - \beta_{vh}^+ S_h^- I_v^+$ in \eqref{doteS} and $y - \beta_{vh}^- S_h^- I_v^+$ in \eqref{dotev2}.
Since $y - {\beta_{vh}}^+ S_h^- I_v^+ = \beta_{vh} S_h I_v - \beta_{vh}^+ S_h^- I_v^+$, 
\begin{subequations}
\begin{equation}\label{yminus}
y - {\beta_{vh}}^+ S_h^- I_v^+ = (\beta_{vh} - \beta_{vh}^+) S_h I_v - {\beta_{vh}}^+ S_h e_v^+ +  {\beta_{vh}}^+ e_{S}^- I_v^+
\end{equation}
and similarly
\begin{equation}\label{yminuss}
y - {\beta_{vh}}^- S_h^- I_v^+ = (\beta_{vh} - \beta_{vh}^-) S_h I_v - {\beta_{vh}}^- S_h e_v^+ +  {\beta_{vh}}^- e_{S}^- I_v^+
\end{equation}
\end{subequations}
Insering \eqref{yminus} in \eqref{doteS} 
and \eqref{yminuss} in \eqref{dotev2}, we get
\begin{eqnarray*}
\dot{e}_S^-
& = &
-\mu_h e_S^- - k_S^- ((\beta_{vh} - {\beta}_{vh}^+) S_h I_v\\
& &
- {\beta_{vh}}^+ S_h e_v^+ +  {\beta_{vh}}^+ e_S^- I_v^+)\\
\dot{e}_v^+
& = & (\beta_{hv}^+ - {\beta}_{hv}) I_h + {\beta}_{hv}^+ e_h^+ - (\beta_{hv}^+ - {\beta}_{hv}) I_v I_h - {\beta}_{hv}^+ I_v e_h^+ \\
& &
- {\beta}_{hv}^+ e_v^+ I_h^+  - \mu_v e_v^+ \\
& &
+ k_v^+ ((\beta_{vh} - {\beta}_{vh}^-) S_h I_v - \beta_{vh}^- S_h e_v^+ + \beta_{vh}^- e_S^- I_v^+) 
\end{eqnarray*}
which finally yields \eqref{eq15a}, together with \eqref{eq16a}.
The proof is the same for system \eqref{eq15b}.

Last, the fact that the matrices $\AOBSa(t), \AOBSb(t)$ are Metzler, and that the vectors $\bOBSa(t), \bOBSb(t)$ are nonnegative and null in the absence of uncertainties, comes directly from the formulas previously proved and the estimates in Lemma \ref{le1} and \ref{le2}.
This achieves the proof of Lemma \ref{le3}.
\hfill$\blacksquare$

\subsection{Proof of Theorem \ref{MainTheorem}}

 We demonstrate here the claimed property for $\VOBSa$, the case of $\VOBSb$, being analogous, will not be treated.
 {\em Throughout this proof we remove the index \eqref{OBS1} from the variables and parameters, in order to simplify the notation.}
 
Defining the vector $u:=\begin{pmatrix} 1 & \rho &\omega \end{pmatrix}\t$, the (state) function $V(t)$ defined in \eqref{Vdefa} writes $V(t) = u\t X(t)$ (recall that the error vector $X$ has been given in \eqref{eq22a}).
Therefore, writing as usual $\dot V$ the derivative of $V$ along the trajectories, one has, due to formula \eqref{eq15a} in Lemma \ref{le3},
\begin{equation}
\label{eq24}
\dot{V}(t)+ \delta (t) V (t)
= u\t (A(t)+\delta(t) I) X(t) + u\t b(t)
\end{equation}
where $I$ denotes de identity matrix.
Now notice that, with $F$ defined in \eqref{Fdefa}, one has $u\t b(t) \leq F(t)$.
We show next that
\begin{equation}
\label{ThmIneq}
u\t (A(t)+\delta(t) I)\leq 0
\end{equation}
Given these two facts, one may deduce from \eqref{eq24} that
\[
\dot V(t) + \delta (t) V(t) \leq F(t)
\]
which gives \eqref{Vestimatesa} by Gronwall's lemma and thus achieves the proof.
It therefore remains to show \eqref{ThmIneq} in order to complete the proof of Theorem \ref{MainTheorem}.

Let us show that \eqref{ThmIneq} holds when the gains are chosen as prescribed in the statement.
Equation \eqref{ThmIneq} can be written as the following system of inequalities, valid for any $t\geq 0$:
\begin{subequations}\label{Thmdelta}
\hspace{-.2cm}
  \begin{align}
\hspace{-.2cm}
   \delta(t) &\leq \mu_h+k_S^-(t)\beta_{vh}^+(t) I_v^+(t) -\omega k_v^+(t)\beta_{vh}^-(t) I_v^+(t),\\
\hspace{-.2cm}
   \label{Thmdelta2}\delta(t) &\leq \mu_h+\gamma-\frac{\omega}{\rho} \beta_{hv}^+(t)(1-I_v(t)),\\
\hspace{-.4cm}
   \delta(t) &\leq \mu_v + k_v^+(t)\beta_{vh}^-(t)S_h(t) +\beta_{hv}^+(t)I_h^+(t)-\frac{k_S^-(t)}{\omega}\beta_{vh}^+(t) S_h(t).
  \end{align}
  \end{subequations}

Using \eqref{ella} and \eqref{xia} in \eqref{Thmdelta}, yields  the following equivalent set of inequalities:
 \begin{subequations}\label{Thmxi}
  \begin{align}
  \label{Thmxi1} \xi(t) I^+_v(t) & \leq {I_v^+(t)}\xi(t),\\
  \label{Thmxi2} \xi(t) I^+_v(t)  &\leq \gamma-\frac{\omega}{\rho} \beta_{hv}^+(t)(1-I_v(t)),\\
   \label{Thmxi3} \xi(t) I^+_v(t)  &\leq {\mu_v-\mu_h} - \frac{\xi(t) {S_h}(t)}{\omega} + {\beta_{hv}^+(t)I_h^+(t)}.
  \end{align}
  \end{subequations}
  Observe that \eqref{Thmxi1} is trivially verified.
  Also, note that 
  $$
  \xi I_v^+ \leq \displaystyle {\gamma-\varepsilon}\leq \gamma-\varepsilon\frac{\beta_{hv}^+(1-I_v)}{\beta_{hv}^+(1-I_v^-)},
  $$ 
  which yields (in view of  \eqref{rhoa}) inequality   \eqref{Thmxi2}.
 \if{ \begin{equation}
  \xi I_v^+ \leq \gamma-\varepsilon\frac{\beta_{hv}^+(1-I_v)}{\beta_{hv}^+(1-I_v^-)} \leq \gamma- \frac{\omega}{\rho} {\beta_{hv}^+(1-I_v)}  ,
  \end{equation} }\fi

Last, by the definition of $\xi(t)$ given in \eqref{xia}, for any $t\geq 0$,
  $$
  \xi(t)\leq \frac{\mu_v-\mu_h+\beta_{hv}^+(t)I_h^+(t)}{S_h^+(t)/\omega+I_v^+(t)},
  $$
  which implies  \eqref{Thmxi3} since $S_h^+(t)\geq S_h(t).$ 
   
   The inequalities in \eqref{ThmIneq} are thus verified for the chosen values of the parameters. This completes the proof of Theorem \ref{MainTheorem}.
\hfill$\blacksquare$

  \begin{remark}
Setting in \eqref{Thmdelta} the gains $k_S^-(t)$ and $k_v^+(t)$ to zero and choosing $\displaystyle\frac{\omega}{\rho}$ small enough, $\delta(t)$ has values at least equal to $\mu_h$ (due to the fact that $\mu_v>\mu_h$).
  Therefore it is always possible to ensure that the convergence speed verifies $\mu_h\leq \delta(t)$.
On the other hand, since the last term in the right-hand side of \eqref{Thmdelta2} is nonpositive, the value of $\delta(t)$ in \eqref{deltaa} necessarily satisfies $\delta(t) < \mu_h + \gamma$ for any $t\geq 0$.
\end{remark}

\bibliographystyle{plain}
\bibliography{ecc2018}

\end{document}